\documentclass[]{article}

\usepackage[utf8]{inputenc}
\usepackage[english]{babel}
\usepackage{pdfpages}
\usepackage[a4paper]{geometry}
\usepackage{graphicx}
\usepackage{amsmath,amsthm,amssymb}
\usepackage{wasysym}
\usepackage{microtype}
\usepackage[colorlinks=false, pdfborder={0 0 0}]{hyperref}
\usepackage[capitalise]{cleveref}
\usepackage[all,cmtip]{xy}
\usepackage{authblk}

\title{Box spaces of the free group that neither contain expanders nor embed into a Hilbert space}
\author{Thiebout Delabie\thanks{Universit\'{e} de Neuch\^{a}tel, \url{thiebout.delabie@unine.ch}, supported by Swiss NSF grant no. 200021 163417} 
\ and Ana Khukhro\thanks{Universit\'{e} de Neuch\^{a}tel, \url{anastasia.khukhro@unine.ch}}}

\date{\today}

\newtheorem*{thm}{Main Theorem}
\newtheorem{theorem}{Theorem}[section]
\newtheorem{lemma}[theorem]{Lemma}
\newtheorem{proposition}[theorem]{Proposition}
\newtheorem{corollary}[theorem]{Corollary}
\newtheorem{definition}[theorem]{Definition}
\newtheorem{remark}[theorem]{Remark}
\theoremstyle{definition}
\newtheorem*{acknow}{Acknowledgements}

\newcommand{\Id}{\operatorname{Id}}

\newcommand{\PSL}[2]{\operatorname{PSL}_{#1}(#2)}
\newcommand{\PGL}[2]{\operatorname{PGL}_{#1}(#2)}
\newcommand{\N}{\mathbb{N}}
\newcommand{\Z}{\mathbb{Z}}

\newcommand{\R}{\mathbb{R}}
\newcommand{\C}{\mathbb{C}}
\newcommand{\Ha}{\mathbb{H}}
\newcommand{\Qo}[2]{F_3/(N_{#1}\cap\Theta(N_{#2}))}
\newcommand{\Ob}[1]{\mathcal{O}\left(#1\right)}
\newcommand{\norm}[1]{\left\|#1\right\|}
\newcommand{\p}{5}
\newcommand{\girth}{\operatorname{girth}}

\parskip\bigskipamount

\begin{document}

\maketitle

\begin{abstract}
We construct box spaces of a free group that do not coarsely embed into a Hilbert space, but do not contain coarsely nor weakly embedded expanders. We do this by considering two sequences of subgroups of the free group: one which gives rise to a box space which forms an expander, and another which gives rise to a box space that can be coarsely embedded into a Hilbert space. We then take certain intersections of these subgroups, and prove that the corresponding box space contains generalized expanders. We show that there are no weakly embedded expanders in the box space corresponding to our chosen sequence by proving that a box space that covers another box space of the same group that is coarsely embeddable into a Hilbert space cannot contain weakly embedded expanders. 
\end{abstract}

\section{Introduction}
Given a residually finite, finitely generated group $G$, we say that a sequence of nested finite index normal subgroups of the group is a \emph{filtration} if this sequence of subgroups has trivial intersection. Given such a filtration $\{N_i\}$ of $G$ and fixing a generating set of $G$, we can consider each finite quotient $G/N_i$ with the Cayley graph metric induced by image the generating set of $G$. 
\begin{definition}
The box space $\Box_{N_i} G$ of $G$ with respect to a filtration $\{N_i\}$ is the disjoint union of the finite quotients $\{G/N_i\}$ with their Cayley graph metrics, with the distance between different quotients defined to be at least the larger of their diameters. 
\end{definition}
Note that different choices of generating set for $G$ give rise to coarsely equivalent box spaces (\cite{Kh12}).
Studying the geometric properties of this space essentially reduces to studying geometric properties that the finite quotients $G/N_i$ have \emph{uniformly}. 
To give an example, let us first introduce the following notion.
\begin{definition}
Given two metric spaces $(X, d_X)$ and $(Y,d_Y)$, we say that a map $f:X \rightarrow Y$ is a coarse embedding if there exist non-decreasing control functions $\rho_{\pm}:\mathbb{R}_+ \rightarrow \mathbb{R}_+$ such that for all $a,b \in X$, we have
$$\rho_-(d_X(a,b))\leq d_Y(f(a),f(b)) \leq \rho_+(d_X(a,b)).$$
\end{definition}
In this way, a box space coarsely embeds into a Hilbert space if and only if all of the components $G/N_i$ admit coarse embeddings into a Hilbert space with the same control functions $\rho_{\pm}$. For brevity, \emph{``embeddable''} is sometimes used to mean coarsely embeddable into a Hilbert space.

An a priori weaker notion is that of a weak embedding. It was used in \cite{Gro} to construct a group which does not admit a coarse embedding into a Hilbert space.
\begin{definition}
Given a sequence of finite metric spaces $(X_n)_{n\in \mathbb{N}}$, and a metric space $Y$, a sequence of maps $f_n:X_n \rightarrow Y$ is a weak embedding if there is $C>0$ such that each $f_n$ is $C$-Lipschitz, and for all $r>0$, we have 
$$\lim_{n\rightarrow \infty} \sup_{x\in X_n} \frac{|f_n^{-1}(B_Y(f_n(x),r))|}{|X_n|}=0,$$
where $B_Y(y,r)$ denotes the ball of radius $r$ about $y\in Y$.
\end{definition}
When the target space $Y$ is of bounded geometry (i.e. the cardinality of balls is uniformly bounded by some constant depending only on the radius), then the above condition is equivalent to 
$$\lim_{n\rightarrow \infty} \sup_{x\in X_n} \frac{|f_n^{-1}(f_n(x))|}{|X_n|}=0.$$
A coarse embedding of a metrized sequence of finite graphs of bounded geometry into a space $Y$ implies a weak embedding of the sequence of graphs into $Y$.

Box spaces are good sources of examples in the world of bounded geometry metric spaces because their geometric properties often have strong links with algebraic or analytic properties of the parent group $G$. For example, the first explicit construction of expander graphs, due to Margulis (\cite{Mar}), was given in the form of a box space of a group with Kazhdan's property (T).

Expanders are sequences of graphs of bounded degree which are highly connected. These seemingly contradictory properties of sparseness and connectivity make them sought-after objects in computer science, network design, and computational group theory, as well as interesting geometric objects in their own right. The connectivity properties of expanders are also what prevents them from embedding well into Banach spaces. A weakly embedded expander is the obstruction to embedding coarsely into a Hilbert space that was first used in the probabilistic proof of \cite{Gro}. Note that it is unknown whether a weakly embedded expander implies the existence of a coarsely embedded expander.

For a long time, the presence of coarsely (or weakly) embedded expanders was in fact the only known obstruction to a bounded geometry metric space coarsely embedding into a Hilbert space. Note that if one does not impose the condition of bounded geometry, then $\ell^p$ with $p>2$ is a space which does not contain expanders, yet does not admit a coarse embedding into a Hilbert space (\cite{JR}).

An important step towards answering the question of whether expanders are indeed the only possible obstruction was the paper of Tessera \cite{Tes}, in which he was able to give a characterization of spaces which do not embed coarsely into a Hilbert space in terms of \emph{generalized expanders}, which satisfy corresponding Poincar\'{e} inequalities relative to a measure.

In the groundbreaking article \cite{AT}, Arzhantseva and Tessera gave examples of sequences of finite Cayley graphs of uniformly bounded degree which do not contain weakly embedded expanders but do not embed coarsely into a Hilbert space. Their examples make use of \emph{relative expanders}, which are a specific case of generalized expanders. One of their examples is a box space of $\mathbb{Z}^2 \rtimes SL(2,\mathbb{Z})$, a group with relative property (T): this box space does not embed coarsely into a Hilbert space because the parent group does not have the Haagerup property, and it does not contain expanders thanks to a proposition (Proposition 2, \cite{AT}) which shows that expanders cannot be embedded into a sequence of group extensions where the sequence of quotients and the sequence of normal subgroups which make up the extension both embed coarsely into a Hilbert space. They also give constructions of box spaces of wreath products, including an example which admits a fibred coarse embedding into a Hilbert space (i.e. it is a box space of a group with the Haagerup property, see \cite{CWW} for the proof of this equivalence). All of these examples are constructed using sequences of finite groups which do coarsely embed into a Hilbert space, and the non-embeddability of the resulting spaces is encoded in the action of one subgroup on another.

The following problem (\cite{AT}, Section 8: Open Problems) remained open: does there exist a sequence of finite graphs with bounded degree and girth (i.e. the length of the smallest cycle) tending to infinity that does not coarsely embed into a Hilbert space but does not contain a weakly embedded expander? 
The original motivation for this question of Arzhantseva and Tessera was the possibility to use such a sequence for the construction of a group with these properties (although the presence of such a sequence would not guarantee that the group constructed would not contain expanders elsewhere). Arzhantseva and Tessera have since constructed such a group without the use of such a sequence of graphs (\cite{AT}), and the question about the existence of such a large girth sequence remained unanswered.
A natural way to construct such a sequence would be to use a box space of a non-abelian free group. This requires a different method to the one used in \cite{AT}, since there are no obvious ``building blocks'' which can be used to construct the sequence (as with the semidirect products of embeddable groups in \cite{AT}). 

In this article, we answer this question by proving the following theorem.

\begin{thm}
\hypertarget{Main}
There exists a filtration of the free group $F_3$ such that the corresponding box space does not coarsely embed into a Hilbert space, but does not admit a weakly embedded expander sequence. 
\end{thm}

The proof of the theorem involves taking a sequence of subgroups which gives rise to an embeddable box space, and intersecting it with one which gives rise to an expander, in a controlled way. 

\begin{acknow}
We warmly thank Alain Valette for his many useful suggestions, and for lending us his personal copy of \cite{RamGr}. 
We would like to thank Romain Tessera for his insightful comments leading to the streamlining of several proofs in the paper.
We also thank Alex Lubotzky for a stimulating discussion, and Goulnara Arzhantseva and Tianyi Zheng for their helpful remarks. Finally, we thank the referee for their thorough reading of the paper which led to many improvements in the exposition. This work was supported in part by the Swiss NSF grant no. 200021 163417.
\end{acknow}

\subsection*{Overview} 

The overall structure of the proof is as follows.
We construct a sequence of subgroups $\{N_i\}$ of $F_3$ which gives rise to expanders (Section \ref{Quat}), and consider the sequence of homology covers of the quotients $\{F_3/N_i\}$; this gives rise to another sequence of subgroups $\Theta(N_i)< N_i$ of $F_3$ (Section \ref{HomCovers}) such that the corresponding quotients of $F_3$ coarsely embed into a Hilbert space. We then consider the quotients of $F_3$ by intersections of these sequences of subgroups, as in the following diagram, where the arrows represent quotient maps.
$$\xymatrix{
&  &  & F_3/\Theta(N_3)\ar[d]\cdots \\
&  & F_3/\Theta(N_2)\ar[d] & \Qo{3}{2}\ar[d]\ar[l]\cdots \\
& F_3/\Theta(N_1)\ar[d] & \Qo{2}{1}\ar[d]\ar[l] & \Qo{3}{1}\ar[d]\ar[l]\cdots \\
\{1\} & F_3/N_1\ar[l] & F_3/N_2\ar[l] & F_3/N_3\ar[l]\cdots \\
}$$
In Section \ref{Voila}, we choose a subsequence of quotients $\{F_3/ (N_{n_i}\cap \Theta(N_{k_i}))\}$ which lie on some path that moves sufficiently slowly away the horizontal (expander) sequence in this ``triangle'' of intersections. 

We do this so that for such a quotient $F_3/ (N_{n_i}\cap \Theta(N_{k_i}))$, we can control the eigenvalues corresponding to those eigenvectors of the Laplacian which are not coming from lifts of eigenvectors of the Laplacian on the quotient $F_3/ (N_{n_i -1}\cap \Theta(N_{k_i}))$ which is horizontally to the left of $F_3/ (N_{n_i}\cap \Theta(N_{k_i}))$ (we do this using representation theory in Section \ref{RepSec}). This ensures, via the results on generalized expanders of Section \ref{GenExpSec} that the chosen sequence will not coarsely embed into a Hilbert space.

On the other hand, each of the quotients $F_3/ (N_{n_i}\cap \Theta(N_{k_i}))$ surjects onto $F_3/ \Theta(N_{k_i})$, and we prove in Section \ref{ExpCov} that such a sequence then cannot contain weakly embedded expanders.\\

\section{Expanders and embeddability into Hilbert spaces}\label{Beginning}
\subsection{Expanders and generalized expanders}\label{GenExpSec}

Let $X=(E,V)$ be a finite, $k$-regular graph, and number the vertices of $X$, $V=\{v_1,v_2,...,v_n\}$. The \emph{adjacency matrix} of $X$ is the matrix $A$ indexed by pairs of vertices $v_i,v_j \in V$ such that $A_{ij}$ is equal to the number of edges connecting $v_i$ to $v_j$. We will restrict ourselves to considering simple graphs, and so for us, this number will always be equal to either 0 or 1. 

The \emph{Laplacian} is defined as the matrix $\Delta:= k\Id - A$, which can be viewed as an operator $\ell^2(V)\rightarrow \ell^2(V)$ (here, $\ell^2$ will always denote the real Hilbert space). If $|V|=n$, then $\Delta$ is an $n\times n$ symmetric matrix and thus, counting multiplicities, has $n$ real eigenvalues,
$$\lambda_0=0 \leq \lambda_1 \leq \cdots \leq \lambda_{n-1}.$$
Note that the corresponding eigenvectors are orthogonal.
The first non-trivial eigenvalue $\lambda_1$ is linked to connectivity properties of the graph $X$, namely via the \emph{Cheeger constant} $h(X):= \inf |\partial F|/|F|$, where the infimum is taken over all subsets $F$ of $X$ satisfying $0<|F|\leq |X|/2$. 
The well-known \emph{Cheeger-Buser inequality} links the first non-trivial eigenvalue of the Laplacian with the Cheeger constant: $\frac{\lambda_1}{2}\leq h(X) \leq \sqrt{2k\lambda_1}$.

While any finite connected graph $X$ has a non-zero Cheeger constant, it is rather difficult to construct a sequence of $k$-regular graphs of growing size such that their Cheeger constants are bounded uniformly away from zero. Given a sequence of $k$-regular graphs $\{X_n\}$ with $|X_n|\rightarrow \infty$, we say that $\{X_n\}$ is an \emph{expander sequence} if there exists an $\varepsilon>0$ such that $h(X_n)>\varepsilon$ for all $n$. 
We now give three characterizations of expanders.

\begin{theorem}
Let $(\mathcal{G}_n)_n$ be a sequence of $k$-regular Cayley graphs. This sequence is an expander if and only if one of the following equivalent statements is true:
\begin{enumerate}
\item\label{1} There exists a $c>0$ such that $h(\mathcal{G}_n)\ge c$ for every $n$.
\item\label{2} There exists an $\varepsilon>0$ such that $\lambda_1(\mathcal{G}_n)\ge\varepsilon$.
\item\label{3} There exists a $C$ such that for every $n$ and every $1$-Lipschitz map $\varphi\colon\mathcal{G}_n\to\ell^2$ we have
\[\sum_{x,y\in\mathcal{G}_n}\norm{\varphi(x)-\varphi(y)}^2\le C|\mathcal{G}_n|^2.\]
\end{enumerate}
\end{theorem}

\begin{proof}
The equivalence $\ref{1}\Leftrightarrow\ref{2}$ is due to the Cheeger-Buser inequality.

The proof of $\ref{2}\Rightarrow\ref{3}$ is based on Proposition 5.7.2 of \cite{LSG}.

Set $C=\frac{k}{\varepsilon}$. Now for any $n$ we can take $v_0=1_{\mathcal{G}_n},v_1,\ldots,v_{|\mathcal{G}_n|-1}$ to be the eigenvectors of the Laplacian $\Delta_n$ on $\mathcal{G}_n$.

Let $f\colon\mathcal{G}_n\to\R$ such that $\displaystyle\sum_{x\in\mathcal{G}_n}f(x)=0$, we can write $f = a_1v_1+\ldots+a_{|\mathcal{G}_n|-1}v_{|\mathcal{G}_n|-1}$.
We have
\begin{eqnarray*}
\sum_{d(x,y)=1}|f(x)-f(y)|^2 & = & \sum_{d(x,y)=1} f(x)(f(x)-f(y)) - f(y)(f(x)-f(y))\\
& = & \sum_{d(x,y)=1} f(x)(f(x)-f(y)) + f(x)(f(x)-f(y))\\
& = & \sum_{d(x,y)=1} 2f(x)(f(x)-f(y))\\
& = & \sum_{x\in\mathcal{G}_n} 2f(x)(\Delta_n(f)(x))\\
& = & 2\langle f, \Delta_n(f)\rangle
\end{eqnarray*}

Using that $\langle v_i,v_j\rangle=0$ if $i\neq j$ we thus have:
\begin{eqnarray*}
\sum_{d(x,y)=1}|f(x)-f(y)|^2 & = & 2\left\langle a_1v_1 + \ldots + a_{|\mathcal{G}_n|-1}v_{|\mathcal{G}_n|-1}, a_1\lambda_1v_1 + \ldots + a_{|\mathcal{G}_n|-1}\lambda_{|\mathcal{G}_n|-1}v_{|\mathcal{G}_n|-1}\right\rangle\\
& = & 2\left(\lambda_1\norm{a_1v_1}^2 + \ldots + \lambda_{|\mathcal{G}_n|-1}\norm{a_{|\mathcal{G}_n|-1}v_{|\mathcal{G}_n|-1}}^2\right)\\
& \ge & 2\lambda_1\norm{f}^2\\
& = & 2\lambda_1(\mathcal{G}_n) \sum_{x\in\mathcal{G}_n}|f(x)|^2.
\end{eqnarray*}
Now let $\varphi\colon \mathcal{G}_n\to\ell^2$ be a $1$-Lipschitz map. Without loss of generality we may assume that $\displaystyle\sum_{x\in\mathcal{G}_n}\varphi(x) = 0$. We can decompose $\varphi$ according to an orthonormal basis. Using this decomposition and the computation above, we find that $\displaystyle 2\lambda_1(\mathcal{G}_n) \sum_{x\in\mathcal{G}_n}\norm{\varphi(x)}^2 \le \sum_{d(x,y)=1}\norm{\varphi(x)-\varphi(y)}^2 \le \sum_{d(x,y)=1} 1 \le k|\mathcal{G}_n|$.

Now we can bound $\displaystyle\sum_{x,y\in\mathcal{G}_n}\norm{\varphi(x)-\varphi(y)}^2$ as follows:
\begin{eqnarray*}
\sum_{x,y\in\mathcal{G}_n} \norm{\varphi(x)-\varphi(y)}^2 & = & \sum_{x,y\in\mathcal{G}_n} \norm{\varphi(x)}^2 + \norm{\varphi(y)}^2 - 2\langle\varphi(x),\varphi(y)\rangle\\
& = & \sum_{x\in\mathcal{G}_n} 2|\mathcal{G}_n|\norm{\varphi(x)}^2 - 2\left\langle\sum_{x\in\mathcal{G}_n}\varphi(x),\sum_{y\in\mathcal{G}_n}\varphi(y)\right\rangle\\
& \le & \frac{k|\mathcal{G}_n|}{\lambda_1(\mathcal{G}_n)}|\mathcal{G}_n|\\
& \le & C|\mathcal{G}_n|^2.\\
\end{eqnarray*}
This proves that $\ref{2}\Rightarrow\ref{3}$.

Now we only have to prove that $\ref{3}\Rightarrow\ref{2}$. Set $\varepsilon=\frac{1}{C}$ and suppose that $\lambda_1(\mathcal{G}_n)<\varepsilon$ for some $n$. Let $f$ be the eigenvector $v_1$ for this $n$. Set $B=\displaystyle\sum_{d(x,y)=1} |f(x)-f(y)|^2$. Now we can take $\varphi\colon \mathcal{G}_n\to\ell^2(\mathcal{G}_n)$ with $\varphi(x)\colon\mathcal{G}_n\to\R\colon y\to \frac{1}{\sqrt{B}}f(y^{-1}x)$. Note that $\mathcal{G}_n$ is a Cayley graph, therefore $y^{-1}x$ is well-defined.
Now $\varphi$ is $1$-Lipschitz because for every $x,y\in\mathcal{G}_n$ with $d(x,y)=1$ we have 
$$\norm{\varphi(x)-\varphi(y)}^2\le \displaystyle\frac{1}{B}\sum_{z\in\mathcal{G}_n} |f(z^{-1}x)-f(z^{-1}y)|^2 \le \frac{1}{B}\sum_{d(x',y')=1} |f(x')-f(y')|^2=1.$$
Showing that $2\varepsilon\displaystyle\sum_{x\in\mathcal{G}_n} |f(x)|^2> B$ would show that $\varphi$ does not satisfy $\displaystyle\sum_{x,y\in\mathcal{G}_n}\norm{\varphi(x)-\varphi(y)}^2\le C|\mathcal{G}_n|^2$, because of the following argument:
\begin{align*}
\sum_{x,y\in\mathcal{G}_n} \norm{\varphi(x)-\varphi(y)}^2 & =  \sum_{x,y\in\mathcal{G}_n} \norm{\varphi(x)}^2 + \norm{\varphi(y)}^2 - 2\langle\varphi(x),\varphi(y)\rangle\\
& =  \sum_{x\in\mathcal{G}_n} 2|\mathcal{G}_n|\norm{\varphi(x)}^2 - 2\left\langle\sum_{x\in\mathcal{G}_n}\varphi(x),\sum_{y\in\mathcal{G}_n}\varphi(y)\right\rangle.
\end{align*}
But we have the following computation
$$\left(\sum_{x\in \mathcal{G}_n}\varphi(x) \right)(z) = \sum_{x\in \mathcal{G}_n} \frac{1}{\sqrt{B}}f(z^{-1}x)= \sum_{y\in \mathcal{G}_n} \frac{1}{\sqrt{B}}f(y)= \frac{1}{\sqrt{B}}\left\langle f, 1_{\mathcal{G}_n} \right\rangle = 0$$
so that we have $$\left\langle\sum_{x\in\mathcal{G}_n}\varphi(x),\sum_{y\in\mathcal{G}_n}\varphi(y)\right\rangle = 0$$ and thus
\begin{align*}
\sum_{x,y\in\mathcal{G}_n} \norm{\varphi(x)-\varphi(y)}^2
& = \sum_{x\in\mathcal{G}_n} 2|\mathcal{G}_n|\norm{\varphi(x)}^2\\
& =  \frac{2}{B}|\mathcal{G}_n|\sum_{x,y\in\mathcal{G}_n} |f(y^{-1}x)|^2\\
& =  \frac{2}{B}|\mathcal{G}_n|^2\sum_{y\in\mathcal{G}_n} |f(y)|^2\\
& \ge  C|\mathcal{G}_n|^2.\\
\end{align*}
To show that $2\varepsilon\displaystyle\sum_{x\in\mathcal{G}_n} |f(x)|^2>B$, we use the first equality obtained at the beginning of the proof of $\ref{2}\Rightarrow\ref{3}$ to make the following computations:
\begin{eqnarray*}
B & = & \sum_{d(x,y)=1} |f(x)-f(y)|^2\\
& = & 2\sum_{x\in\mathcal{G}_n} f(x)(\Delta_n(f)(x))\\
& = & 2\sum_{x\in\mathcal{G}_n} \lambda_1(\mathcal{G}_n)f(x)f(x)\\
& = & 2\lambda_1(\mathcal{G}_n)\sum_{x\in\mathcal{G}_n} |f(x)|^2\\
& < & 2\varepsilon\sum_{x\in\mathcal{G}_n} |f(x)|^2.\\
\end{eqnarray*}
This concludes the proof.
\end{proof}

The following definition of Tessera \cite{Tes} was introduced in order to characterize the failure to coarsely embed into a Hilbert space.

\begin{definition}[Definition 1, \cite{Tes}]
Let $(\mathcal{G}_n)_n$ be a sequence of graphs with $d_n$ the induced graph metric on $\mathcal{G}_n$. This sequence is said to be a generalized expander if there exists a sequence $r_n$ with $r_n\to\infty$ as $n\to\infty$, a sequence of probability measures $\mu_n$ on $\mathcal{G}_n\times\mathcal{G}_n$ supported on $\{(x,y) \in \mathcal{G}_n\times\mathcal{G}_n : d_n(x,y)\geq r\}$  and a constant $C>0$ such that for every $1$-Lipschitz map $\varphi\colon (\mathcal{G}_n)_n\to\ell^2$ we have the following condition:
\[\sum_{x,y\in\mathcal{G}_n}\norm{\varphi(x)-\varphi(y)}^2\mu_n(x,y) \le C.\]
\end{definition}

In particular, expanders in the usual sense are generalized expanders. 
It is proved in \cite{Tes} that a metric space does not embed coarsely into a Hilbert space if and only if it contains a coarsely embedded sequence of generalized expanders.
In \cite{AT}, Arzhantseva and Tessera define the notions of expansion relative to subgroups, partitions, and measures, to differentiate between different cases of generalized expansion, and give examples of box spaces which do not coarsely embed into a Hilbert space and do not contain coarsely (and even weakly) embedded expanders.

We now give a natural way to find generalized expanders, which coincides with the special case of expansion relative to subgroups. For this, we will need the notion of \emph{lifts of eigenvectors} of the Laplacian. Given a finite group $G$ and its Cayley graph with respect to some generating set $S$, a quotient $H$ of $G$ considered with the induced metric, and an eigenvector $v\in \ell^2(H)$ of the Laplacian on $H$, its lift to $\ell^2(G)$ is the vector $\widetilde{v}$ defined by $\widetilde{v}(g)=v(\pi(g))$, where $\pi$ is the quotient map $G \rightarrow H$. It is easy to check that the lift $\widetilde{v}$ is an eigenvector of the Laplacian on $G$ if the degree of the Cayley graphs of $G$ and $H$ is the same (i.e. if the ball of radius 1 in $H$ and its lift in $G$ are the same).

\begin{proposition}\label{genExp}
Let $r_n$ be a sequence such that $r_n\to\infty$ as $n\to\infty$.
Let $G_n$ be a sequence of finite $k$-generated groups with their corresponding Cayley graphs, and let $H_n$ be a sequence of quotient groups of $G_n$ with the induced metrics such that the kernel $N_n$ of $G_n\to H_n$ is non-trivial, but $B_{G_n}(e,r_n)\cap N_n = \{e\}$.

If there exists a constant $\varepsilon>0$ such that for every eigenvector of the Laplacian $\Delta_n$ on $G_n$ that is not the lift of an eigenvector of the Laplacian of $H_n$, the corresponding eigenvalue is bigger than $\varepsilon$. Then the Cayley graphs of $G_n$ form a generalized expander.
\end{proposition}

\begin{proof}
Take $D=|G_n|(|N_n|-1)$ and take $\mu$ such that $\mu(x,y)$ is equal to $\frac{1}{D}$ if $x^{-1}y$ lies in $N_n\setminus\{e\}$ and $0$ otherwise. Take $C = \frac{2k}{\varepsilon}$.


Now for any $n$ we can take $M_{N_n}$ to be the averaging operator, i.e. $M_{N_n}(f)(x) = \frac{1}{|N_n|}\displaystyle\sum_{g\in N_n} f(gx)$. 
The space generated by the lifts of eigenvectors of the Laplacian of $H_n$ is equal to the image of $M_{N_n}$, so the space generated by all other eigenvectors is the image of $\Id-M_{N_n}$. These eigenvectors correspond to eigenvalues bigger than $\varepsilon$, so $\Delta_n(1-M_N) \ge \varepsilon (1-M_N)$, where $\Delta_n$ denotes the Laplacian on $G_n$.
Note that as $N_n$ is a normal subgroup we have for every $s\in S$, every $f\in \ell^2(G_n)$ and every $x\in G_n$, writing $\lambda_s$ for the left regular representation of $s$, that $\displaystyle (\lambda_s\circ M_{N_n})(f)(x) = M_{N_n}(f)(s^{-1}x) = \frac{1}{|N_n|}\sum_{g\in N_n} f(gs^{-1}x) = \frac{1}{|N_n|}\sum_{h\in N_n} f(s^{-1}hx) = (M_{N_n}\circ \lambda_s)(f)(x)$. So as $\displaystyle\Delta_n = k\Id - \sum_{s\in S}\lambda_s$ we have that $M_{N_n}$ commutes with $\Delta_n$.

As $M_{N_n}$ is an orthogonal projection and the Laplacian $\Delta_n$ is a positive self-adjoint operator, we can conclude that $\Delta_n \ge  (1-M_N) \Delta_n =  \Delta_n(1-M_N) \ge \varepsilon (1-M_N)$.

So we can make the following computation:
\begin{eqnarray*}
\sum_{d(x,y)=1}|f(x)-f(y)|^2 & = & \sum_{d(x,y)=1} f(x)(f(x)-f(y)) - f(y)(f(x)-f(y))\\
& = & \sum_{d(x,y)=1} 2f(x)(f(x)-f(y))\\
& = & \sum_{x\in G_n} 2f(x)(\Delta_n(f)(x))\\
& = & 2\langle f, \Delta_n(f)\rangle\\
& \ge & 2\varepsilon\langle f, (\Id - M_{N_n}) f\rangle\\
& = & 2\varepsilon \sum_{x\in G_n} f(x) \left( f(x) - \frac{1}{|N_n|} \sum_{z\in N_n}f(xz)\right)\\
\end{eqnarray*}

Now let $\varphi\colon \mathcal{G}_n\to\ell^2$ be a $1$-Lipschitz map. We can decompose $\varphi$ according to an orthonormal basis. 
Using this decomposition we find the following:
\[\displaystyle 2\varepsilon \sum_{x\in G_n}\left\langle\varphi(x),\varphi(x)-\frac{1}{|N_n|}\sum_{z\in N_n}\varphi(xz)\right\rangle \le \sum_{d(x,y)=1}\|\varphi(x)-\varphi(y)\|^2 \le \sum_{d(x,y)=1} 1 \le k|G_n|.\]
Now we can bound $\displaystyle\sum_{x,y\in G_n}\|\varphi(x)-\varphi(y)\|^2$ as follows:
\begin{eqnarray*}
\sum_{x,y\in G_n} \|\varphi(x)-\varphi(y)\|^2\mu(x,y) & = & 2\sum_{x,y\in G_n} \left(\|\varphi(x)\|^2 - \langle\varphi(x),\varphi(y)\rangle\right)\mu(x,y)\\
& = & \frac{2}{D}\sum_{x\in G_n}\sum_{z\in N_n\setminus\{e\}} \left(\|\varphi(x)\|^2 - \langle\varphi(x),\varphi(xz)\rangle\right)\\
& = & \frac{2}{D}\sum_{x\in G_n}\sum_{z\in N_n} \left(\|\varphi(x)\|^2 - \langle\varphi(x),\varphi(xz)\rangle\right)\\
& \le & \frac{2}{D}\sum_{x\in G_n} \left(|N_n|\|\varphi(x)\|^2 - \sum_{z\in N_n}\langle\varphi(x),\varphi(xz)\rangle\right)\\
& \le & \frac{2|N_n|}{D}\sum_{x\in G_n} \left( \|\varphi(x)\|^2 - \left\langle\varphi(x), \frac{1}{|N_n|}\sum_{z\in N_n}\varphi(xz)\right\rangle \right)\\
& \le &  \frac{2|N_n|}{D}\cdot\frac{k|G_n|}{2\varepsilon}\\
& \le &  \frac{2k}{\varepsilon}\\
& = &  C.\\
\end{eqnarray*}
Note that the second-to-last inequality follows from the inequality $\frac{|N_n| \cdot |G_n|}{D} \leq \frac{|N_n|}{|N_n|-1} \leq 2$.
Therefore we can conclude that $G_n$ is a generalized expander.
\end{proof}

\subsection{Expanders and finite covers}\label{ExpCov}

Proposition 2 of \cite{AT} states that given a sequence of short exact sequences of finite groups $\{N_n \rightarrow G_n \rightarrow Q_n\}_n$ such that the quotient groups $\{Q_n\}$ and the subgroups $\{N_n\}$ coarsely embed into a Hilbert space (with respect to the induced metrics from $\{G_n\}$), the sequence $\{G_n\}$ cannot contain weakly embedded expanders.

We now show that the assumption on the subgroups is satisfied if the sequences $\{Q_n\}$ and $\{G_n\}$ both approximate the same group, i.e. if they are both box spaces of the same infinite group.

\begin{proposition}\label{NoExp}
Let $G$ be a finitely generated, residually finite group with a filtration $\{N_n\}$, and let $\{M_n\}$ be another sequence of finite index normal subgroups of $G$ such that $N_n > M_n$ for all $n$. If $\Box_{(N_n)}G$ coarsely embeds into a Hilbert space, then $\Box_{(M_n)}G$ does not contain weakly (and thus coarsely) embedded expanders. 
\end{proposition}

\begin{proof}
Consider the sequence of short exact sequences $$\{N_n/M_n \rightarrow G/M_n \rightarrow G/N_n\}_n,$$
where $G/M_n$ and $G/N_n$ are considered with the metric induced by the restriction of the respective box space metrics, and $N_n/M_n$ is considered with the metric induced by viewing $N_n/M_n$ as a subspace of $G/M_n$. 

Since both $\Box_{(M_n)}G$ and $\Box_{(N_n)}G$ are box spaces of $G$ and $G/N_n$ is a quotient of $G/M_n$, for all $R$ there is some $m(R)$ such that for all $n\geq m(R)$, the balls of radius $R$ in $G/M_n$ and $G/N_n$ are isometric to balls of radius $R$ in $G$; moreover, the quotient map $\pi_n: G/M_n \rightarrow G/N_n$ is an isometry when restricted to a ball of radius $R$. This means that the ball of radius $R$ in $G/M_n$ does not contain any non-trivial element of $N_n/M_n$, and so we see that the $N_n/M_n$ are \emph{sparse} with respect to the subspace metric, i.e. there exists a sequence $r_n \rightarrow \infty$ such that any two points of $N_n/M_n$ are at distance at least $r_n$ from each other.

We can deduce from this that the sequence $(N_n/M_n)$ coarsely embeds into a Hilbert space: indeed, consider the embedding of each $N_n/M_n$ into $\ell^2(N_n/M_n)$ defined by sending each element $x$ of $N_n/M_n$ to $r_n \chi_x$, that is, the characteristic function of $x$ in $\ell^2(N_n/M_n)$ scaled by $r_n$.

Thus, since we also assume that the box space $\Box_{(N_n)}G$ coarsely embeds into a Hilbert space, this sequence of short exact sequences satisfies the assumptions of Proposition 2 of \cite{AT}, and so the box space $\Box_{(M_n)}G$ does not contain a weakly embedded expander.
\end{proof}

\section{Subgroups of the free group}

To construct box spaces of the free group which do not admit a coarse embedding into a Hilbert space without containing weakly embedded expanders, we will use two sequences of subgroups of the free group: one which gives rise to a box space which is an expander, and one which does admit a coarse embedding into a Hilbert space. We then use information about these two sequences to prove that the box space obtained using certain intersections of these subgroups has the desired properties. In the following two subsections, we will describe the two sequences of subgroups.

\subsection{Constructing subgroups of $F_3$}\label{Quat}
In this section we will define a sequence of nested finite index normal subgroups $N_n$ of the free group $F_3$. We will rely heavily on the machinery described in \cite{Lub} to construct a sequence of Ramanujan graphs, and will frequently refer to relevant results and proofs in \cite{Lub}.

We fix the prime $p=5$, noting that $p\equiv 1 \bmod 4$, and an odd prime $q\neq 5$ such that $-1$ is a quadratic residue modulo $q$ and $5$ is a quadratic residue modulo $2q$. Such a prime exists, for example $q=29$ (in fact, there exist infinitely many such primes, see for example the proof of Theorem 4.6 in \cite{DK}). 

Consider $\Ha(\Z)$, the integer quaternions, with the equivalence relation $a\sim b$ if there exists $m,n\in\mathbb{N}$ such that $\p^n a = \pm \p^m b$. 
Note that the equivalence relation $\sim$ is compatible with multiplication in $\Ha(\Z)$.
Recall that the norm $N$ on $\Ha(\Z)$ is defined by $N(\alpha)= \alpha \bar{\alpha}$, where $\bar{\alpha}$ is the quaternion conjugate to $\alpha$. Abusing the notation, we will also write $\alpha$ for the equivalence class of $\alpha$ with respect to $\sim$. Note that for elements $\alpha\in \Ha(\Z)/\sim$ with $N(\alpha)=\p^m$ for some $m\in\Z$, we have $\alpha^{-1}=\bar{\alpha}$.

\begin{proposition}[Corollary 2.1.11, \cite{Lub}]\label{F2}
The subgroup $\Lambda(2)$ of $\Ha(\Z)/\sim$ generated multiplicatively by the set $S_\p:= \{ 1+ 2i, 1+ 2j, 1+ 2k\}$ is the free group $F_3$ on the set $S_\p$. 
\end{proposition}

An equivalent way to see this free group is as in Section 7.4 of \cite{Lub}. 
Consider the group $\Gamma=\Ha(\mathbb{Z}[\frac{1}{p}])^{\times}/Z(\Ha(\mathbb{Z}[\frac{1}{p}])^{\times})$, where $Z$ denotes the center. 
Following the notation of \cite{Lub}, we can define a sequence of subgroups of $\Gamma$ by 
$\Gamma(n):= \ker(\Gamma \rightarrow \Ha(\mathbb{Z}[\frac{1}{p}]/n\mathbb{Z}[\frac{1}{p}])^{\times}/Z(\Ha(\mathbb{Z}[\frac{1}{p}]/n\mathbb{Z}[\frac{1}{p}])^{\times}))$. 
The subgroup $\Gamma(2)$ is generated by the image of the set $S_\p^{\pm}:=\{ 1\pm 2i, 1\pm 2j, 1\pm 2k\}$ and is exactly the free group $\Lambda(2)$ above.

The following theorem of \cite{Lub} tells us that we can construct quotients of the free group which are expanders. Recall that a \emph{Ramanujan graph} is a $k$-regular graph such that all of its eigenvalues apart from 0 and possibly $2k$ lie in the interval $[k-2\sqrt{k-1}, k+2\sqrt{k-1}]$ (thus a family of Ramanujan graphs achieves the best possible spectral gap). 

\begin{theorem}\label{LubRam}[Theorem 7.4.3, \cite{Lub}]
Let $p\equiv 1 \bmod 4$ be a prime, and let $N=2M$ be an integer such that $(M,2p)=1$. Assume that there is $\varepsilon\in \mathbb{Z}$ such that $\varepsilon^2 \equiv -1 \bmod M$. Consider the set $S^{\pm}_p$ of the $p+1$ solutions $x_0+x_1 i + x_2 j + x_3 k$ of $x_0^2 + x_1^2 +x_2^2 +x_3^2= p$ (where $x_0>0$ is odd, and $x_1, x_2, x_3$ are even). Associate to each element $x_0+x_1 i + x_2 j + x_3 k$ of $S_p$ the matrix $\begin{bmatrix}x_0 + x_1 \varepsilon & x_2 +x_3 \varepsilon\\-x_2 + x_3 \varepsilon & x_0 - x_1 \varepsilon\end{bmatrix} \bmod M$ in $\PGL{2}{M}$. 
Then the image of the group generated by $S^{\pm}_p$ under this map is the quotient $\Gamma(2)/\Gamma(N)$, which is isomorphic to $\PSL{2}{M}$ if $p$ is a quadratic residue modulo $N$, and the Cayley graph of $\Gamma(2)/\Gamma(N)$ with respect to the image of $S^{\pm}_p$ is a non-bipartite Ramanujan graph. 
\end{theorem}

We will apply this theorem to a particular sequence of subgroups of our free group $\Gamma(2)$ (to which we will from now on refer to simply as $F_3$), namely the sequence of subgroups $N_n:= \Gamma(2q^n)$. 

We have chosen $q$ such that $-1$ is a quadratic residue modulo $q$, and now we show that it is also a quadratic residue modulo $q^n$ for any $n$.

\begin{proposition}\label{QuadRes}
Let $q$ be an odd prime. For every $u\in\Z$ and every $n\in\N$, if $u$ is a quadratic residue modulo $q$ and $u$ is not zero modulo $q$, then $u$ is a quadratic residue modulo $q^n$.
\end{proposition}
\begin{proof}
By induction we may assume that there exists a number $b$ such that $b^2\equiv u\bmod q^{n-1}$. So there exists a number $c$ such that $b^2=u+cq^{n-1}$. Now take $a=b-tcq^{n-1}$, where $t$ is the inverse of $2b$ modulo $q$ ($2b$ is invertible modulo $q$ since $u$ is not zero modulo $q$). Now $a^2=b^2-2btcq^{n-1}+ t^2c^2q^{2n-2}\equiv u+cq^{n-1}-cq^{n-1}=u \bmod q^n$.
\end{proof}

We note that this implies that there exists a $q$-adic integer $\varepsilon$ such that $\varepsilon^2 = -1$. We can thus use this $\varepsilon$ to define the map in \cref{LubRam} so that the maps are compatible for $M$ equal to different powers of $q$. 

Similarly, since we chose $q$ so that $5$ will be a quadratic residue modulo $2q$, it is also a quadratic residue modulo $2q^n$ for all $n$.

\begin{proposition}
Let $q$ be an odd prime. For every $u\in\Z$ and every $n\in\N$, if $u$ is a quadratic residue modulo $2q$ and $u$ is not zero modulo $q$, then $u$ is a quadratic residue modulo $2q^n$.
\end{proposition}
\begin{proof}
By induction, we can assume that there is a number $b$ such that $b^2\equiv u\bmod 2q^{n-1}$. So there exists a number $c$ such that $b^2=u+2cq^{n-1}$. Now take $a=b-tcq^{n-1}$, where $t$ is the inverse of $b$ modulo $q$ ($b$ is invertible modulo $q$ since $u$ is not zero modulo $q$). Now $a^2=b^2 -2btcq^{n-1}+ t^2c^2q^{2n-2}\equiv u+2cq^{n-1} -2cq^{n-1}=u \bmod 2q^n$.
\end{proof}

Thus, the assumptions of \cref{LubRam} are satisfied, and we obtain the following.

\begin{corollary}\label{isoSL}
For any $n\in\N$ we have that $F_3/N_n$ is isomorphic to $\PSL{2}{q^n}$.
\end{corollary}

We will now investigate the properties of certain intermediate quotients $N_k/N_n$. We first need the following lemma.

\begin{lemma}\label{abelianplus}
Let $k,n\in\N$ with $0<k\le n\le 2k$. Then the kernel $\ker(\PSL{2}{q^{n}}\rightarrow \PSL{2}{q^{k}})$ of reduction modulo $q^k$ is isomorphic to $\Z_{q^{n-k}}^3$.
\end{lemma}

\begin{proof}
We have that $\ker(\PSL{2}{q^{n}}\rightarrow \PSL{2}{q^{k}})$ is equal to $\left\{B\in\PSL{2}{q^n}\mid B\equiv I_2\bmod q^k \right\}$, where $I_2$ denotes the 2-by-2 identity matrix.
This is precisely the set of matrices of the form $\begin{bmatrix}1+aq^k & bq^k \\ cq^k & 1-aq^k\end{bmatrix}$ with $a$, $b$ and $c$ in $\Z_{q^{n-k}}$. For every two such matrices, we find
\[\begin{bmatrix}1+aq^k & bq^k \\ cq^k & 1-aq^k\end{bmatrix}\begin{bmatrix}1+a'q^k & b'q^k \\ c'q^k & 1-a'q^k\end{bmatrix}\equiv\begin{bmatrix}1+aq^k+a'q^k & bq^k+b'q^k \\ cq^k+c'q^k & 1-aq^k-a'q^k\end{bmatrix}\bmod q^n.\]
Thus we have that $\ker(\PSL{2}{q^{n}}\rightarrow \PSL{2}{q^{k}})$ is isomorphic to $\Z_{q^{n-k}}^3$.
\end{proof}

\begin{corollary}\label{abelian}
Let $k,n\in\N$ with $0<k\le n\le 2k$. Then $N_k/N_n$ is isomorphic to $\Z_{q^{n-k}}^3$.
\end{corollary}

\begin{proof}
The map in \cref{LubRam} which provides the isomorphism between $\Gamma(2)/\Gamma(2q^n)$ and $\PSL{2}{q^n}$ commutes with the map of reduction modulo $q^k$, and thus we have that $N_k/N_n \cong \Gamma(2q^k)/\Gamma(2q^n) \cong \ker(\PSL{2}{q^{n}}\rightarrow \PSL{2}{q^{k}}) \cong \Z_{q^{n-k}}^3$. 
\end{proof}

\begin{remark}\label{MatrixGen}
A fact that will be of direct use to us later is that, since every element of $N_{n-1}/N_n$ can be viewed as a matrix of the form $\begin{bmatrix}1+cq^{n-1} & dq^{n-1} \\ fq^{n-1} & 1-cq^{n-1}\end{bmatrix}$ with $c$, $d$ and $f$ in $\Z_{q}$, a generating set of $N_{n-1}/N_n$ can be given by the matrices 
$$\begin{bmatrix}1+q^{n-1} & 0 \\ 0 & 1-q^{n-1}\end{bmatrix}, \begin{bmatrix}1 & q^{n-1} \\ 0 & 1\end{bmatrix}, \begin{bmatrix}1 & 0 \\ q^{n-1} & 1\end{bmatrix}.$$

We note that for any $k<n-1$, the matrix $\begin{bmatrix}1+q^{n-1} & 0 \\ 0 & 1-q^{n-1}\end{bmatrix}$ is equivalent to the matrix 
$$\begin{bmatrix} q^{2n-2}+q^{n-1}+1&-q^{n+k}\\ q^{2n-k-3}&-q^{n-1}+1\end{bmatrix},$$ 
which is the commutator of the matrices $\begin{bmatrix}1 & q^{k+1} \\ 0 & 1\end{bmatrix}$ and $\begin{bmatrix}1 & 0 \\ q^{n-k-2} & 1\end{bmatrix}$.
\end{remark}

\subsection{Homology covers}\label{HomCovers}

Given a finite graph $X$, one can construct a covering graph $\widetilde{X}$ of $X$ such that $\widetilde{X}$ is the cover corresponding to the the quotient $\pi_1(X) \rightarrow \bigoplus^{r} \mathbb{Z}_m$ of highest rank $r$ possible. Indeed, since $\pi_1(X)$ is a free group, the rank $r$ is simply the rank of this free group. 

We recall that as a graph, the cover can be viewed in the following way. 
First, choose a maximal spanning tree $T$ of the graph $X$.
Construct the Cayley graph of $\bigoplus^{r} \mathbb{Z}_m$ with respect to the image of the free generating set of $\pi_1(X)$.
Note that the free generating set of $\pi_1(X)$ is in bijection with its image in $\bigoplus^{r} \mathbb{Z}_m$, and also in bijection with the edges of $X$ not contained in the maximal tree $T$. Let $\kappa$ be the bijection between the edges not in $T$ and this generating set of $\bigoplus^{r} \mathbb{Z}_m$.

The vertices of the cover $\widetilde{X}$ will be in bijection with the vertices of $T\times \bigoplus^{r} \mathbb{Z}_m$, and we will think of them as copies of $T$ indexed by elements of $\bigoplus^{r} \mathbb{Z}_m$. The edges of $\widetilde{X}$ are defined as follows: if two vertices $v$ and $w$ in $X$ are connected by an edge $e$ which is not in $T$, then given such a vertex $\widetilde{v}$ in one of the copies of $T$ indexed by an element $a$ of $\bigoplus^{r} \mathbb{Z}_m$, we connect it via an edge to a vertex $\widetilde{w}$ in the copy of $T$ indexed by the element $a\kappa(e)$ of $\bigoplus^{r} \mathbb{Z}_m$; if two vertices $v$ and $w$ in $X$ are connected by an edge $e$ which is in $T$, then we connect $\widetilde{v}$ and $\widetilde{w}$ by an edge in the same copy of $T$.

The covering space $\widetilde{X}$ obtained in this way is called the $m$\emph{-homology cover of }$X$. 

The situation we are interested in is as follows: we have a sequence of quotients of the free group, $\{F_3/N_n\}$, metrized using the Cayley graph metric coming from the free generating set of $F_3$, and we consider the sequence of their $q$-homology covers, with $q$ as in the previous subsection. By covering space theory (for details, see for example \cite{Kh13}), the $q$-homology covers of the $F_3/N_n$ are also quotients of $F_3$, by the subgroups 
$$\Theta(N_n):= N_n^q [N_n,N_n]$$
where $N_n^q$ denotes the subgroup $\left\langle g^q : g\in N_n \right\rangle$ of $F_3$ generated by all the $q$th powers of elements of $N_n$. 
Since $\Theta(N_n)< N_n$, we have that $\cap_n N_n = \{1\}$ implies that $\cap_n \Theta(N_n) = \{1\}$ and so we can consider the box space $\Box_{\Theta(N_n)} F_3$. 

The box space of a free group corresponding to a $2$-homology cover was first considered by Arzhantseva, Guentner and \v{S}pakula in \cite{AGS}, who proved that such a box space coarsely embeds into a Hilbert space, as one can construct a wall structure on it using the covering space structure. 

In \cite{Kh13}, this was generalized as follows. 
\begin{theorem}[Theorem 4, \cite{Kh13}]
Given any $m\geq 2$ and any sequence $\{X_n\}$ of $2$-connected finite graphs where the number of maximal spanning trees in $X_n$ not containing a given edge does not depend on the edge, the sequence of $\mathbb{Z}_m$-homology covers of the $X_n$ coarsely embeds into a Hilbert space (uniformly with respect to $i$) if $\girth(X_n)\rightarrow \infty$.
\end{theorem}

Note that this holds even if the sequence $\{X_n\}$ does not embed coarsely into a Hilbert space. 
In particular, we have that the box space $\Box_{\Theta(N_n)} F_3$ corresponding to the $q$-homology covers of any box space $\Box_{N_n} F_3$ of the free group embeds coarsely into a Hilbert space, even if the box space $\Box_{N_n} F_3$ is an expander sequence. 

We now restrict ourselves to the following setting: the sequence $\{N_n\}$ is as defined in the previous subsection, and we consider the sequence of subgroups $\{\Theta(N_n)\}$ corresponding to the $q$-homology covers. We have the following relation between the sequences, which we will need in the subsequent sections.

\begin{proposition}\label{ntok}
Let $k,n\in\N$ with $0<k<n$. Then $N_n\Theta(N_k) = N_{k+1}$.
\end{proposition}

\begin{proof}
We will prove this proposition by induction on $n-k$ for fixed $k$. For $n=k+1$ we clearly have that $N_{k+1} < N_{k+1} \Theta(N_k)$. So  to prove the base case of the induction, it suffices to show that $\Theta(N_k)=N_k^q[N_k,N_k]< N_{k+1}$.
We will in fact show that $N_k^q< N_{k+1}$ and $[N_k,N_k]< N_{k+1}$. 

To see that $N_k^q< N_{k+1}$, take an element $x\in N_k < F_3$. Up to the equivalence relation $\sim$, we can assume that $x$ has the form $x=1+aq^k+bq^ki+cq^kj+dq^kk$. 
Then we can make the following computation:
\begin{align*}
x^q & = \left(1+aq^k +bq^ki+cq^kj+dq^kk\right)^q\\
& = 1 +  q^{k+1}(a+bi+cj+dk)  + \dfrac{q(q-1)}{2}q^{2k}(a+bi+cj+dk)^2+\ldots\\
& \equiv 1 & \bmod q^{k+1}\\
\end{align*}
and so we have that $x^q \in N_{k+1}$ and thus $N_k^q< N_{k+1}$. 

We also have that $[N_k,N_k]< N_{k+1}$, since the quotient $N_k/N_{k+1}$ is abelian by \cref{abelian}. 
Therefore we have $\Theta(N_k)\subset N_{k+1}$, and this proves the proposition for $n=k+1$.

Now by induction we may assume that $N_{n-1}\Theta(N_k)=N_{k+1}$. As $N_n< N_{n-1}$ we have that $N_n\Theta(N_k)< N_{k+1}$.

We have that $N_{n-1}\Theta(N_k)> N_n\Theta(N_k)$. It suffices now to show that $N_{n-1}\Theta(N_k)< N_n\Theta(N_k)$, or equivalently that $N_{n-1}\Theta(N_k)/ N_n\Theta(N_k)$ is trivial.
Due to the second isomorphism theorem we have that 
$$\dfrac{N_{n-1}\Theta(N_k)}{ N_n\Theta(N_k)} = \dfrac{N_{n-1} N_n\Theta(N_k)}{ N_n\Theta(N_k)} \cong \dfrac{N_{n-1}}{ N_{n-1}\cap N_n\Theta(N_k)}.$$

Now this is a quotient of $N_{n-1}/N_n$ since $N_{n-1}\cap N_n\Theta(N_k)>N_n$, and is therefore isomorphic to a quotient of $\Z_q^3$ as a consequence of \cref{abelian}. 
So it suffices to take a generating set of $N_{n-1}/N_n$ and show that the elements of this generating set lie in $N_n\Theta(N_k)$ modulo $N_n$. This will ensure that the quotient $N_{n-1}/ (N_{n-1}\cap N_n\Theta(N_k))$ of $N_{n-1}/N_n$ is trivial.

In fact it suffices to show that the generating elements lie in $\Theta(N_{n-2})$ modulo the subgroup $N_n$, since $\Theta(N_{n-2})< N_n\Theta(N_k)$. 

Due to \cref{isoSL} we can view $N_{n-1}/N_n$ as a subgroup of $\PSL{2}{q^n}$.
As in \cref{MatrixGen}, an example of a generating set of $N_{n-1}/N_n$ is 
$$\left\{\begin{bmatrix} q^{n-1}+1&0\\ 0&-q^{n-1}+1\end{bmatrix}, \begin{bmatrix} 1&q^{n-1}\\ 0&1\end{bmatrix},\begin{bmatrix} 1&0\\ q^{n-1}&1\end{bmatrix}\right\}.$$ 
Now modulo $q^n$, 
\begin{align*}
\begin{bmatrix} q^{n-1}+1&0\\ 0&-q^{n-1}+1\end{bmatrix} &\equiv \begin{bmatrix} q^{n-2}+1&0\\ 0&-q^{n-2}+1\end{bmatrix}^q \\
\begin{bmatrix} 1&q^{n-1}\\ 0&1\end{bmatrix} &\equiv \begin{bmatrix} 1&q^{n-2}\\ 0&1\end{bmatrix}^q \\
\begin{bmatrix} 1&0\\ q^{n-1}&1\end{bmatrix} &\equiv\begin{bmatrix} 1&0\\ q^{n-2}&1\end{bmatrix}^q.
\end{align*}
Thus all elements of $N_{n-1}/N_n$ lie in $N_n\Theta(N_k)/N_n$ so $N_{n-1}/(N_{n-1}\cap N_n\Theta(N_k))$ is trivial and therefore $N_n\Theta(N_k)=N_{n-1}\Theta(N_k)=N_{k+1}$.
\end{proof}

\subsection{Representation theory}\label{RepSec}
The aim of this section is to study representations of the quotients $F_3/ (N_n \cap \Theta(N_k))$ for certain values of $n$ and $k$. 

All representations of $\Qo{n-1}{k}$ can be lifted to representations of $\Qo{n}{k}$. In this section we want to show that the dimensions of the representations of $\Qo{n}{k}$ which are not such lifts grow like $q^n$ for $k$ fixed. \footnote{Alain Valette pointed out to us that a proof of this can also be given using the Mackey machine.}

For $k,n\in\N$ with $0<2k\le n$ define $B_{k,n}$ as follows: 
$$B_{k,n}:=\left\{\begin{bmatrix}a&b\\0&a^{-1}\end{bmatrix}\in N_k/N_n \mid a\in\Z_{q^n}^\times, b\in\Z_{q^n}\right\}.$$
Another way of stating this condition on $a\in\Z_{q^n}^\times$ and $b\in\Z_{q^n}$ is that $a\equiv 1 \bmod q^k$ and $b\equiv 0 \bmod q^k$.
Note that $B_{k,n}$ is a subgroup of $N_k/N_n$.
In fact, for every such choice of $a$ and $b$, we have that  $\begin{bmatrix}a&b\\0&a^{-1}\end{bmatrix}$ is an element of $N_k/N_n$ and
we thus see that $B_{k,n}$ has order $(q^{n-k})^2=q^{2n-2k}$.

\begin{lemma}\label{allRep}
Let $k,n,l\in\N$ with $0<2k\le 2k+l\le n$. Then every irreducible representation $\pi$ of $B_{k,n}$ for which $\pi\left(\begin{bmatrix}1&q^{n-l}\\0&1\end{bmatrix}\right)=\Id$ and $\pi\left(\begin{bmatrix}1&q^{n-l-1}\\0&1\end{bmatrix}\right)\neq\Id$ has dimension $q^{n-2k-l}$.
\end{lemma}

\begin{proof}
If $n=2k$, then $l=0$ and due to \cref{abelian} we know that $B_{k,n}$ is abelian.
In this case all irreducible representations of $B_{k,n}$ have dimension $1$, which satisfies this proposition.

For other values of $k$ and $n$, we will now consider the irreducible representations of $B_{k,n}$.

Take $\omega = \displaystyle e^\frac{2\pi i}{q^{n-k}}$.
As $k\ge 1$ we have that $1+q^k$ is of order $q^{n-k}$ in $\Z_{q^n}^\times$ and therefore generates $\{\alpha\in\Z_{q^n}^\times \mid \alpha\equiv 1 \bmod q^k\}$.
Now for every $j\in\{0,1,\ldots,q^{k}-1\}$ define $\rho_j\colon B_{k,n} \to \C $ by
$$\begin{bmatrix} (1+q^k)^\beta&b \\ 0&(1+q^k)^{-\beta} \end{bmatrix}  \mapsto \omega^{\beta j}.$$

For every such $j$ with $j\not\equiv 0 \bmod q$ set $V_j$ to be the finite-dimensional Hilbert space with $\{\xi_x \mid x\equiv j\bmod q^k, x\in\Z_{q^{n-k}}\}$ as orthogonal basis, where $\xi_x$ denotes the sequence indexed by elements of $\mathbb{Z}_{q^{n-k}}$ which takes the value 1 at $x\in \mathbb{Z}_{q^{n-k}}$ and 0 elsewhere.
Let $\pi_j$ be the representation of $B_{k,n}$ on $V_j$ such that 
$$\pi_j\left(\begin{bmatrix} a&b\\ 0&a^{-1}\end{bmatrix}\right)\xi_x=e^\frac{2\pi iabx}{q^{n}}\xi_{a^2x}.$$
\\

Now we can calculate the characters of these representations:
\begin{eqnarray*}
\chi_{\rho_j}\left(\begin{bmatrix} a&b\\ 0&a^{-1}\end{bmatrix}\right) & = & \rho_j\left(\begin{bmatrix} a&b\\ 0&a^{-1}\end{bmatrix}\right)\\
\chi_{\pi_j}\left(\begin{bmatrix} a&b\\ 0&a^{-1}\end{bmatrix}\right) & = & \sum_{x\equiv j \bmod q^{k}} \left\langle\xi_x , \pi_j\left(\begin{bmatrix} a&b\\ 0&a^{-1}\end{bmatrix}\right)\xi_x \right\rangle\\
& = & \sum_{x\equiv j \bmod q^{k}} \left\langle\xi_x , e^\frac{2\pi iabx}{q^{n}}\xi_{a^2x} \right\rangle\\
& = & \sum_{x\equiv j \bmod q^{k}} e^\frac{2\pi iabx}{q^{n}}\left\langle\xi_x , \xi_{a^2x} \right\rangle
\end{eqnarray*}
Note that if $a\equiv 1 \bmod q^{k}$ and $a^{2} \equiv 1 \bmod q^{n-k}$, then $a \equiv 1 \bmod q^{n-k}$.
Thus, if $a\not\equiv 1 \bmod q^{n-k}$, then for every $x\in\Z_{q^{n-k}}$ we have $\left\langle\xi_x , \xi_{a^2x} \right\rangle=0$, so $\chi_{\pi_j}\left(\begin{bmatrix} a&b\\ 0&a^{-1}\end{bmatrix}\right) = 0$.
If $b\not\equiv 0 \bmod q^{n-k}$, then $\displaystyle\sum_{x\equiv j \bmod q^{k}} e^\frac{2\pi iabx}{q^{n}}=0$, so $\chi_{\pi_j}\left(\begin{bmatrix} a&b\\ 0&a^{-1}\end{bmatrix}\right) = 0$.
If $a\equiv 1 \bmod q^{n-k}$ and $b\equiv 0 \bmod q^{n-k}$, then $a^2x\equiv x \bmod q^{n-k}$ and $\displaystyle\sum_{x\equiv j \bmod q^{k}} e^\frac{2\pi iabx}{q^{n}}=q^{n-2k}e^\frac{2\pi ijb}{q^{n}}$.
Now for every $j,j'\in\{0,\ldots,q^k-1\}$ with $j\not\equiv 0\bmod q$ we can compute $\langle\chi_{\pi_j\otimes\rho_{j'}},\chi_{\pi_j\otimes\rho_{j'}}\rangle$ using the fact that $|\chi_{\rho_j}(g)|=1$ for every $g\in B_{k,n}$:
\begin{eqnarray*}
\langle\chi_{\pi_j\otimes\rho_{j'}},\chi_{\pi_j\otimes\rho_{j'}}\rangle & = & \frac{1}{|B_{k,n}|}\sum_{a\equiv 1, b\equiv 0 \bmod q^{k}} \left|\chi_{\rho_j}\left(\begin{bmatrix} a&b\\ 0&a^{-1}\end{bmatrix}\right)\chi_{\pi_j}\left(\begin{bmatrix} a&b\\ 0&a^{-1}\end{bmatrix}\right)\right|^2\\
& = & \frac{1}{q^{2n-2k}}\sum_{a\equiv 1, b\equiv 0 \bmod q^{n-k}} \left|q^{n-2k}e^\frac{2\pi ijb}{q^{n}}\right|^2\\
& = & \frac{1}{q^{2n-2k}}q^{2k}q^{2n-4k}\\
& = & 1.\\
\end{eqnarray*}
Varying $j$ and $j'$, we find $q^{2k}-q^{2k-1}$ irreducible representations of dimension $q^{n-2k}$. Note that all of these representations are different.

For every irreducible representation $\pi$ of $B_{k,n-1}$, we can lift this to an irreducible representation $\widetilde{\pi}$ of $B_{k,n}$. We can now consider the (also irreducible and pairwise distinct) representations $\widetilde{\pi}\otimes \rho_j$, for $j\in \{0,1,\ldots,q-1\}$, $\pi$ running through irreducible representations of $B_{k,n-1}$. 
For these representations we have that $\widetilde{\pi}\otimes\rho_j\left(\begin{bmatrix} 1&q^{n-1}\\ 0&1\end{bmatrix}\right) = \Id$, since the matrix $\begin{bmatrix} 1&q^{n-1}\\ 0&1\end{bmatrix}$ lies in $N_{n-1}$ and thus is trivial in $B_{k,n-1}$.

Now we can check if we have found all irreducible representations of $B_{k,n}$:
\begin{eqnarray*}
\sum_{\pi \text{ rep. of} B_{k,n}}\left|\chi_{\pi}( I_2)\right|^2 & = & \sum_{j = 0}^{q^k-1}\sum_{j' \not\equiv 0\bmod q} \left|\chi_{\pi_j\otimes\rho_{j'}}( I_2)\right|^2 + \sum_{j = 0}^{q-1}\sum_{\pi \text{ rep. of }B_{k,n-1}} \left|\chi_{\widetilde{\pi}\otimes\rho_{j}}( I_2)\right|^2\\
& = & \sum_{j = 0}^{q^k-1}\sum_{j' \not\equiv 0 \bmod q} q^{2n-4k} + \sum_{j = 0}^{q-1}\sum_{\pi \text{ rep. of }B_{k,n-1}} \left|\chi_{\widetilde{\pi}}( I_2)\right|^2\\
& = & (q^{2k}-q^{2k-1}) q^{2n-4k} + \sum_{j = 0}^{q-1}q^{2n-2k-2}\\
& = & q^{2n-2k} - q^{2n-2k-1} + q^{2n-2k-1}\\
& = & \left| B_{k,n}\right|.\\
\end{eqnarray*}
Thus we have found all the irreducible representations of $B_{k,n}$. 

By induction we may assume that the proposition is true for $B_{k,n-1}$. 
If $l=0$, then all the irreducible representations of $B_{k,n}$ where the image of $\begin{bmatrix} 1&q^{n-1}\\ 0&1\end{bmatrix}$ is not the identity have dimension $q^{n-2k}$, as they are necessarily those representations not arising as $\widetilde{\pi} \otimes \rho_j$ with $\pi$ an irreducible representation of $B_{k,n-1}$, i.e. they are those representations of the form $\pi_j \otimes \rho_{j'}$ constructed above.

If $l>0$, then all irreducible representations where the image of $\begin{bmatrix} 1&q^{n-l}\\ 0&1\end{bmatrix}$ is the identity, but the image of $\begin{bmatrix} 1&q^{n-l-1}\\ 0&1\end{bmatrix}$ is not, are of the form $\widetilde{\pi}\otimes\rho_{j}$ where $\widetilde{\pi}$ is the lift of an irreducible representation $\pi$ of $B_{k,n-1}$. 
This is because if we consider the other representations, which are of the form $\pi_j\otimes \rho_{j'}$, considering where the vector $\xi_1$ is mapped by $\pi_j\left(\begin{bmatrix} 1&q^{n-l}\\ 0&1\end{bmatrix}\right)$, we see that the image of $\begin{bmatrix} 1&q^{n-l}\\ 0&1\end{bmatrix}$ cannot be equal to the identity.

Now due to the induction hypothesis we have that the dimension of $\pi$ is $q^{(n-1)-2k-(l-1)}=q^{n-2k-l}$. Now the representation $\widetilde{\pi}\otimes\rho_{j}$ has the same dimension, which completes the proof of the lemma.
\end{proof}

\begin{proposition}\label{boundRep}
Let $k,n\in\N$ be such that $3k\le n-1$, then every representation of $\Qo{n}{k}$ that is not the lift of a representation of $\Qo{n-1}{k}$ has dimension at least $q^{n-3k-3}$.
\end{proposition}

\begin{proof}
First note that $\Theta(N_k)/(N_n\cap\Theta(N_k))$ is isomorphic to $N_{k+1}/N_n$: 
\begin{align*}
\Theta(N_k)/(N_n\cap\Theta(N_k)) & \cong (N_n \Theta(N_k))/N_n\\
& \cong N_{k+1}/N_n.
\end{align*}
We have used the second isomorphism theorem and \cref{ntok}. Let us call this isomorphism $\Psi$, 
$$\Psi: \Theta(N_k)/(N_n\cap\Theta(N_k)) \rightarrow N_{k+1}/N_n.$$
We can thus view $N_{k+1}/N_n$ as a subgroup of $\Qo{n}{k}$, via $\Psi$.

Let $\pi$ be a representation of $\Qo{n}{k}$ that is not the lift of a representation of $\Qo{n-1}{k}$. 
This means that $\pi$ is non-trivial on the kernel of the map 
$$\Qo{n}{k} \rightarrow \Qo{n-1}{k}.$$ 
This kernel is equal to $(N_{n-1}\cap \Theta(N_k)) / (N_{n}\cap \Theta(N_k))$. Considering this kernel, we see that it is in fact isomorphic to $N_{n-1}/N_n$:

\begin{align*}
(N_{n-1}\cap \Theta(N_k)) / (N_{n}\cap \Theta(N_k)) &\cong (N_{n-1}\cap \Theta(N_k)) / (N_{n}\cap(N_{n-1} \cap \Theta(N_k)))\\
& \cong ((N_{n-1}\cap\Theta(N_k))N_n) / N_n\\
&\cong (N_{n-1}N_n \cap \Theta(N_k)N_n) / N_n\\
& \cong (N_{n-1} \cap N_{k+1}) / N_n\\
& \cong N_{n-1}/N_n.
\end{align*}
Here, we have used the fact that the $N_i$ are nested, the second isomorphism theorem, \cref{ntok}, and that $n$ is sufficiently larger than $k$.
Let us call this isomorphism $\Phi$,
$$\Phi: (N_{n-1}\cap \Theta(N_k)) / (N_{n}\cap \Theta(N_k)) \rightarrow N_{n-1}/N_n.$$
Now the isomorphisms $\Psi$ and $\Phi$ are compatible, in the sense that $\Phi$ is just a restriction of $\Psi$.
This means that when we restrict the representation $\pi$ to $N_{k+1}/N_n$ (viewed as a as a subgroup of $\Qo{n}{k}$, via $\Psi$), this restriction is  non-trivial on $N_{n-1}/N_n$ as $\pi$ is not a lift.
This implies that at least one of the following elements of $N_{n-1}/N_n$ has an image under $\pi$ that is not the identity: 
$$\begin{bmatrix} 1+q^{n-1}& 0\\ 0&1-q^{n-1}\end{bmatrix},\begin{bmatrix} 1&q^{n-1}\\ 0&1\end{bmatrix}, \begin{bmatrix} 1&0\\ q^{n-1}&1\end{bmatrix}.$$
This is because, as in \cref{MatrixGen}, these matrices generate $N_{n-1}/N_n$. 
The matrix $\begin{bmatrix} 1+q^{n-1}& 0\\ 0&1-q^{n-1}\end{bmatrix}$ is equivalent to $\begin{bmatrix} q^{2n-2}+q^{n-1}+1&-q^{n+k}\\ q^{2n-k-3}&-q^{n-1}+1\end{bmatrix}$, which is the commutator of $\begin{bmatrix} 1&q^{k+1}\\ 0&1\end{bmatrix}$ and $\begin{bmatrix} 1&0\\ q^{n-k-2}&1\end{bmatrix}$, as we have seen in \cref{MatrixGen}, and so the images of both of these must be non-trivial, if the image of their commutator is non-trivial.
The transpose-inverse map is an automorphism, and thus, we may assume without loss of generality that $\pi\left(\begin{bmatrix} 1&q^{n-k-2}\\ 0&1\end{bmatrix}\right)\neq\Id$ (since one of the other two generators have non-trivial images, this also implies that this matrix has a non-trivial image).

Let $B$ be the subgroup corresponding to upper triangular matrices of $N_{k+1}/N_n$ under the isomorphism $\Psi$ between $\Theta(N_k)/(N_n\cap\Theta(N_k))$ and $N_{k+1}/N_n$.
Due to \cref{allRep} we know that $\pi\vert_{B}$ contains a representation of dimension at least $q^{n-3k-3}$ (considering $B_{k+1,n}$ and $l=k+1$). 
Thus we can conclude that $\pi$ has dimension at least $q^{n-3k-3}$.
\end{proof}

\section{Box spaces of the free group}\label{Voila}
In this section we will prove that there exist box spaces of the free group $F_3$ that do not coarsely embed into a Hilbert space, but do not contain weakly embedded expanders either.
To do so we will use the following diagram, made up of quotients of the free group $F_3$ by intersections of the subgroups $N_i$ with the subgroups $\Theta(N_k)$ coming from the $q$-homology covers of the $F_3/N_k$ (see Figure 1).
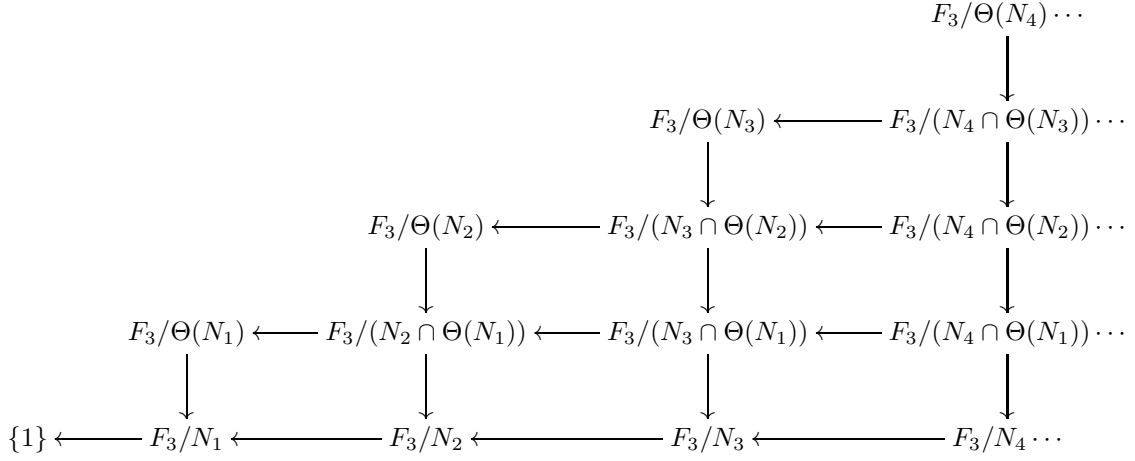
\begin{figure}\label{MagicTriangle}
\xymatrix{
 &  & & & F_3/\Theta(N_4)\cdots\ar[d]\\
 &  &  & F_3/\Theta(N_3)\ar[d] & \Qo{4}{3}\cdots\ar[d]\ar[l]\\
 &  & F_3/\Theta(N_2)\ar[d] & \Qo{3}{2}\ar[d]\ar[l] & \Qo{4}{2}\cdots\ar[d]\ar[l]\\
 & F_3/\Theta(N_1)\ar[d] & \Qo{2}{1}\ar[d]\ar[l] & \Qo{3}{1}\ar[d]\ar[l] & \Qo{4}{1}\cdots\ar[d]\ar[l]\\
\{1\} & F_3/N_1\ar[l] & F_3/N_2\ar[l] & F_3/N_3\ar[l] & F_3/N_4\cdots\ar[l]\\
}
\caption{The magic triangle}
\end{figure}

Note that the quotients $F_3/N_i$ appearing along the bottom row are expanders by \cref{isoSL} and the result of Lubotzky (\cref{LubRam}). 

Set $f_{n,k}(m) = \#\{g\in N_n\cap\Theta(N_k) : |g|\le m\}$ and set $A_{n,k} = [F_3:(N_n\cap\Theta(N_k))]$.

\begin{lemma}\label{sqrt}
If $a^2\equiv b^2\bmod q^n$ and $q\nmid b$, then $a\equiv\pm b\bmod q^{n}$.
\end{lemma}
\begin{proof}
We will prove this lemma by induction. For $n=2$ the lemma follows from Exercise 1 in Section 4.3 of \cite{RamGr}: given that $a^2\equiv b^2\bmod q^2$, we have $a^2-b^2\equiv (a-b)(a+b)\equiv 0\bmod q^2$; this implies that either $a-b\equiv 0\bmod q^2$ or $a+b\equiv 0\bmod q^2$, since it cannot be that both $a-b$ and $a+b$ are divisible by $q$, as that would imply that $q$ divides $b$. 
For bigger $n$, we have that $a^2\equiv b^2 \bmod q^n$ implies $a^2\equiv b^2 \bmod q^{n-1}$, so by induction we have that $a\equiv\pm b\bmod q^{n-1}$. Therefore there exists a $c\in\Z_q$ such that $a\equiv cq^{n-1}\pm b$ modulo $q^n$.

Now it suffices to show that $c\equiv 0\bmod q$. We have that $b^2\equiv a^2\equiv b^2\pm2cbq^{n-1} \bmod q^n$, so $q\mid 2cb$. As $q$ is prime, either $q\mid c$ or $q\mid 2b$. As $q\nmid b$, we have that $q\mid c$ and therefore $a\equiv\pm b\bmod q^{n}$.
\end{proof}

\begin{lemma}\label{boundLoops}
For any $k,n,m\in\N$ with $m$ even, we have $f_{n,k}(m) = \Ob{\frac{\p^{\frac{13}{12}m}}{q^{3n}}+\frac{\p^{\frac{7}{12}m}}{q^n}}$.
\end{lemma}

\begin{proof}
Clearly it suffices to prove the lemma for $k=0$.
For $k=0$, $N_0$ is the whole of the free group $F_3$, and so we are looking to bound the quantity $f_{n,0}(m)=\#\{g\in N_n\cap\Theta(F_3) : |g|\le m\}$. We note that
$$\#\{g\in N_n\cap\Theta(F_3) : |g|\le m\} \leq \#\{g\in N_n : |g|\le m\}.$$
We thus need to estimate from above the number of reduced words of length at most $m$ in $N_n$.

By Corollary 2.6.14 of \cite{RamGr}, for $\alpha \in\mathbb{H}(\Z)$ with $N(\alpha)=5^m$ such that the image of $\alpha$ in $\mathbb{H}(\Z)/\sim$ (which we will denote here by $[\alpha]$) is in $F_3$, there is a unique factorization of the form $\alpha=\pm 5^r w_{m-2r}$, where $w_{m-2r}$ is a reduced word of length $m-2r$ in the elements of the generating set $S_5$. 
Thus in $\mathbb{H}(\Z)/\sim$, such an $[\alpha]$ in $N_n$ will be a reduced word of length $m-2r$, and moreover, each such word corresponds to two quaternions of norm $5^m$ (see Lemma 4.4.2 of \cite{RamGr}). We therefore have 
$$\#\left\{\alpha \in\mathbb{H}(\Z) \mid [\alpha]\in N_n, N(\alpha) = \p^m\right\}=2\#\{g\in N_n : |g|\le m\}.$$

We thus obtain
$$f_{n,0}(m) \leq \#\left\{\alpha \in\mathbb{H}(\Z) \mid [\alpha]\in N_n, N(\alpha) = \p^m\right\} = \#\left\{a + q^n(bi + cj + dk) \mid a^2 + q^{2n}(b^2 + c^2 + d^2) = \p^m\right\}.$$
Now $a^2\equiv \p^m \bmod q^{2n}$, so due to \cref{sqrt} we have $a\equiv \pm \p^\frac{m}{2} \bmod q^{2n}$. This leaves at most $\frac{4\cdot \p^\frac{m}{2}}{q^{2n}}+2$ possibilities for $a$.

Now due to Corollary 2.2.13 of \cite{RamGr} we know that for any fixed $\varepsilon>0$ we have that $\#\{(a,b,c) \mid a^2+b^2+c^2 = x\} = \Ob{x^{\frac{1}{2}+\varepsilon}}$. So we find a bound for $f_{n,0}(m)$:
\begin{eqnarray*}
f_{n,0}(m)  & \le & \sum_{a} \Ob{\left( \frac{5^m-a^2}{q^{2n}}  \right)^{\frac{1}{2}+\varepsilon} }\\
& \le & \left(\frac{4\cdot \p^\frac{m}{2}}{q^{2n}}+2\right) \Ob{\left(\frac{\p^m}{q^{2n}}\right)^{\frac{1}{2}+\varepsilon}}\\
& \le & \Ob{\left(\frac{4\cdot \p^\frac{m}{2}}{q^{2n}}+2\right)\left(\frac{\p^{m(\frac{1}{2}+\varepsilon)}}{q^{n}}\right)}\\
& = & \Ob{\frac{\p^{m(1+\varepsilon)}}{q^{3n}}+\frac{\p^{m(\frac{1}{2}+\varepsilon)}}{q^{n}}}.
\end{eqnarray*}
Now for $\varepsilon = \frac{1}{12}$ we find $f_{n,k}(m) = \displaystyle\Ob{\frac{\p^{\frac{13}{12}m}}{q^{3n}}+\frac{\p^{\frac{7}{12}m}}{q^n}}$.\\
\end{proof}

\begin{theorem}\label{expander}
There exists $N>0$ such that for every $k,n\in\N$ with $n\ge N$, $18<18(k+1)\le n$ and $A_{n,k}\le q^{\frac{19}{6}n}$, we have that every eigenvalue $\lambda$ of the adjacency operator $A$ of $\Qo{n}{k}$ such that some corresponding eigenvector is not the lift of an eigenvector of the adjacency operator of $\Qo{n-1}{k}$ satisfies $\lambda\le \p^\frac{71}{72}+\p^\frac{1}{72}<6$.
\end{theorem}

\begin{proof}

Without loss of generality we may assume that $\lambda\ge 2\sqrt{\p}$.
Take $\theta_j$ such that $\mu_j = 2\sqrt{\p}\cos(\theta_j)$, where $\mu_j$ are the eigenvalues of the adjacency operator $A$. 
As shown in Lemma 4.4.2 of \cite{RamGr}, $f_{n,k}(m)$ is the same as the number of non-backtracking paths from the identity to itself of length at most $m$ in $F_3/N_n$. 
Thus, due to the trace formula given in Corollary 1.4.7 of \cite{RamGr}, we have
\[f_{n,k}(m) = \frac{1}{A_{n,k}} \, \p^\frac{m}{2}\sum_{j=0}^{A_{n,k}-1}\frac{\sin (m+1)\theta_j}{\sin \theta_j}.\]
Take $\psi_j=i\theta_j$. If $|\mu_j|\le 2\sqrt{\p}$, then $\theta_j$ is real and $\left|\frac{\sin (m+1)\theta_j}{\sin \theta_j}\right|\le(m+1)$, and if $|\mu_j|\ge 2\sqrt{\p}$, then $\psi_j$ is real and $\frac{\sin (m+1)\theta_j}{\sin \theta_j} = \frac{\sinh (m+1)\psi_j}{\sinh \psi_j}\ge 0$. So we find the following inequality for any $l$, and in particular for $\mu_l=\lambda$:
\[\frac{A_{n,k}}{\p^\frac{m}{2}}f_{n,k}(m)= \sum_{j=0}^{A_{n,k}-1}\frac{\sin (m+1)\theta_j}{\sin \theta_j}\ge M(\lambda)\frac{\sinh (m+1)\psi_l}{\sinh \psi_l}-(m+1)A_{n,k},\]
where $M(\lambda)$ denotes the multiplicity of the eigenvalue $\lambda$. 
When we take $m$ to be the biggest even integer such that $\p^\frac{m}{2}\le q^{3n}$, we can use \cref{boundLoops} and the fact that we chose $A_{n,k}\le q^{\frac{19}{6}n}$ to obtain the following:
\begin{eqnarray*}
(m+1)A_{n,k}+\frac{A_{n,k}}{ \p^\frac{m}{2}}f_{n,k}(m) & \le & q^{\frac{19}{6}n}\left(m+1 + \Ob{\frac{\p^{\frac{7}{12}m}}{q^{3n}}+\frac{\p^{\frac{1}{12}m}}{q^n}}\right)\\
& \le & q^{\frac{19}{6}n}\left(6n\log_5(q)+1 + \Ob{\frac{q^{\frac{7}{2}n}}{q^{3n}} + \frac{q^{\frac{1}{2}n}}{q^n}}\right)\\
& \le & q^{\frac{19}{6}n}\left(6n\log_5(q)+1 + \Ob{q^{\frac{1}{2}n}+q^{\frac{-1}{2}n}}\right)\\
& = & \Ob{q^{\frac{22}{6}n}}.\\
\end{eqnarray*}

Let $V_\lambda$ be the eigenspace of $A$ corresponding to $\lambda$ on $F_3/(N_n\cap\Theta(N_k))$. Note that $V_{\lambda}$ is a representation space of the group $F_3/(N_n\cap\Theta(N_k))$ (see for example Exercise 4 in Section 4.1 of \cite{RamGr}).
Since some eigenvector is not a lift from $F_3/(N_{n-1}\cap\Theta(N_k))$, the representation of $F_3/(N_n\cap\Theta(N_k))$ on $V_\lambda$ is not a lift from a representation of $F_3/(N_{n-1}\cap\Theta(N_k))$.
We thus have $M(\lambda)\ge q^{n-3k-3}$ due to \cref{boundRep}.

We also have
\[\frac{\sinh (m+1)\psi_l}{\sinh \psi_l}\ge\frac{e^{(m+1)|\psi_l|}}{e^{|\psi_l|}}>e^{(6n\log_\p(q) - 2)|\psi_l|}=\frac{q^{\frac{6n}{\log(\p)}|\psi_l|}}{e^{- 2|\psi_l|}}.\]
We assumed $\lambda\ge 2\sqrt{\p}$, so $\psi_l\ge 0$.
As $e^{2\psi_l}$ is bounded by $e^{\sqrt{\p}}$ we have the following:
\[q^{n-3k-3+\frac{6n}{\log(\p)}\psi_l}\le e^{\sqrt{\p}} M(\lambda)\frac{\sinh (m+1)\psi_l}{\sinh \psi_l} = \Ob{q^{\frac{22}{6}n}}\]
So for big $n$ we find $n-3k-3+\frac{6n}{\log(\p)}\psi_l \le \frac{45}{12}n$. As $18(k+1)\le n$ we see that $n-\frac{n}{6}+\frac{6n}{\log(\p)}\psi_l \le \frac{45}{12}n$. So $\frac{6}{\log(\p)}\psi_l \le \frac{35}{12}$ and therefore $\psi_l \le \frac{35}{72}\log(\p)$. Now we can compute $\lambda$ as follows:
\[\lambda = 2\sqrt{\p}\cos(\theta_l) = 2\sqrt{\p}\cosh(\psi_l) \le \sqrt{\p}\left(\p^\frac{35}{72}+\p^\frac{-35}{72}\right) = \p^\frac{71}{72}+\p^\frac{1}{72}<6.\]
This proves the theorem.
\end{proof}

\begin{corollary}\label{Result}
Let $k_i$ and $n_i$ be non-decreasing sequences in $\N$ with $n_i$ increasing, $18<18(k_i+1)\le n_i$ and $A_{n_i,k_i}\le q^{\frac{19}{6}n_i}$. Then $\Box_{N_{n_i}\cap\Theta(N_{k_i})} F_3$ does not coarsely embed into a Hilbert space.
\end{corollary}

\begin{proof}
We want to apply \cref{genExp}, so we need to check that all the hypotheses hold.

Due to \cref{expander} we know there exists an $N>0$ such that for all $n_i\ge N$, for eigenvalues $\lambda$ of the adjacency operator $A$ of $\Qo{n_i}{k_i}$ such that the corresponding eigenvector is not the lift of an eigenvector of the adjacency operator of $\Qo{n_i-1}{k_i}$, we have that $\lambda\le \p^\frac{35}{36}+\p^\frac{1}{36}$.

Since the Laplacian $\Delta$ is in this case equal to $6\Id - A$ we have that every non-trivial eigenvalue of the Laplacian is greater than $6-\p^\frac{35}{36}-\p^\frac{1}{36}$. 
The quotients $\Qo{n_i-1}{k_i}$ and $\Qo{n_i}{k_i}$ look like $F_3$ (and thus like each other) on bigger and bigger balls, so there exists a sequence $r_i$ such that $r_i\to\infty$ as $i\to\infty$ with $$B(e,r_i)\cap\Big(N_{n_i-1}\cap\Theta(N_{k_i}))/(N_{n_i}\cap\Theta(N_{k_i})\Big)=\{e\},$$ where $B(e,r_i)$ denotes the ball of radius $r_i$ about the identity in $\Qo{n_i}{k_i}$. 
But on the other hand, due to the isomorphism $\Phi$ given as part of the proof of \cref{boundRep}, and \cref{abelian}, we have $N_{n_i-1}\cap\Theta(N_{k_i})\neq N_{n_i}\cap\Theta(N_{k_i})$, since $(N_{n_i-1}\cap\Theta(N_{k_i}))/ (N_{n_i}\cap\Theta(N_{k_i})) \cong N_{n_i -1}/N_{n_i} \cong \mathbb{Z}^3_q$.

Now \cref{genExp} can be applied to the subsequence of $\Qo{n_i}{k_i}$ with $n_i\ge N$. So $\Box_{N_{n_i}\cap\Theta(N_{k_i})} F_3$ contains a generalized expander and therefore does not coarsely embed into a Hilbert space, by the characterization of Tessera \cite{Tes}.
\end{proof}

The \hyperlink{Main}{Main Theorem} now follows from the following result. 

\begin{theorem}\label{MainResult}
There exist increasing sequences $k_i$ and $n_i$ in $\N$ such that $18<18(k_i+1)\le n_i$ and $A_{n_i,k_i}\le q^{\frac{19}{6}n_i}$, and for such $n_i, k_i$, the box space $\Box_{N_{n_i}\cap\Theta(N_{k_i})} F_3$ does not coarsely embed into a Hilbert space, but does not contain weakly embedded expanders.
\end{theorem}

\begin{proof}
Let us first check that such a sequence $(n_i, k_i)$ exists. We have, using the information we have obtained in Sections \ref{Quat} and \ref{HomCovers} about the sizes of quotients in Figure 1, 
\begin{align*}
A_{n_i,k_i} & = |\Qo{n_i}{k_i}|\\
& \leq |F_3/N_{n_i}|\cdot |F_3/\Theta(N_{k_i})|\\
& \leq |F_3/N_{n_i}|\cdot |F_3/N_{k_i}|\cdot |N_{k_i}/\Theta(N_{k_i})|\\
& \leq |\PSL{2}{q^{n_i}}| \cdot |\PSL{2}{q^{k_i}}| \cdot |\mathbb{Z}_q^{2|\PSL{2}{q^{k_i}}|+1}|\\
& \leq q^4\cdot  q^{3(n_i-1)} \cdot q^4\cdot  q^{3(k_i-1)} \cdot q^{2(q^{4}q^{3(k_i-1)})+1}\\
& = q^{3n_i + 3k_i + 3 + 2q^{3k_i +1}}.
\end{align*}

This means that we need $3k_i + 3 + 2q^{3k_i +1}$ to be less than or equal to $\frac{1}{6}n_i$ in order to satisfy $A_{n_i,k_i}\leq q^{\frac{19}{6}n_i}$. Now it is clear that for a sequence of large enough $n_i$ we can take a sequence of $k_i$ which will simultaneously satisfy this and the condition $18<18(k_i+1)\le n_i$. By taking subsequences if necessary, we can ensure that the sequences $n_i$ and $k_i$ are increasing.

\cref{Result} gives us the first part of the statement. For the second part of the statement, we can now apply \cref{NoExp} to the box space $\Box_{N_{n_i}\cap\Theta(N_{k_i})} F_3$ and the box space $\Box_{\Theta(N_{k_i})} F_3$, which is coarsely embeddable into a Hilbert space by the main result in \cite{Kh13} (described in \cref{HomCovers}) as it is a sequence of $q$-homology covers of the graphs $F_3/N_{k_i}$ which satisfy the necessary conditions. 
\end{proof}

\section{Questions and remarks}\label{Qs}
\begin{itemize}
\item
A consequence of \cref{MainResult} is that it is possible to have two box spaces of the same group with respect to \emph{meshed} sequences of subgroups (i.e. sequences of subgroups $\{H_i\}$ and $\{K_i\}$ with $H_1 > K_1 > H_2 > K_2 > H_3 ...$) such that one of the box spaces coarsely embeds into a Hilbert space, and the other does not.

Indeed, after passing to a subsequence, we can find a box space $\Box_{N_{n_j}\cap\Theta(N_{k_j})} F_3$, that can be nested with the box space $\Box_{\Theta(N_{k_j})} F_3$ (this corresponds to taking subgroups which form a sequence of ``steps'' in our diagram of subgroup intersections). We know from \cite{Kh13} that $\Box_{\Theta(N_{k_j})} F_3$ embeds coarsely into $\ell^2$, while $\Box_{N_{n_j}\cap\Theta(N_{k_j})} F_3$ does not by the above.


\item
We have shown that one can choose a sequence in the triangle of intersections (Figure 1) which does not coarsely embed into a Hilbert space, but does not contain weakly embedded expanders. This sequence lies on a path that lies ``close enough'' to the horizontal expander sequence. In an upcoming paper, the first author proves that the horizontal sequences in such a triangle (or, more generally, covers of expanders of uniformly bounded degree) form an expander sequence. What can be said of other sequences in the triangle? Is there a relationship between $k_i$ and $n_i$ for the quotients $\Qo{n_i}{k_i}$ which guarantees coarse embeddability into a Hilbert space? Note that in \cite{DK}, it is shown that two box spaces $\Box_{N_n} G$ and $\Box_{M_n} G$ with $M_n > N_n$ and $[M_n:N_n]$ uniformly bounded independently of $n$ need not be coarsely equivalent.

\item
Let $\{N_n\}$ be a different sequence of subgroups of the free group which gives rise to an expander. Can one prove similar results?

\item
Do homology covers of quotients of non-free groups coarsely embed into a Hilbert space? If $G$ is a finitely generated group, not necessarily free, what can one say about the box space corresponding to the inductively defined sequence of subgroups $N_1:=G$, $N_{n+1}:=N_n^q [N_n, N_n]$? If this box space embeds, can one recreate the triangle argument in such a case?

\item
It is unknown whether there exists a bounded geometry metric space which does not coarsely embed into Hilbert space, but does coarsely embed into $\ell^p$ for some $p>2$. Note that such a space cannot contain coarsely embedded expanders. The box space constructed in this paper does not embed coarsely into $\ell^p$, by the same argument as in the proof of Proposition 4 of \cite{AT}.

\end{itemize}


\end{document}